\documentclass[noams]{compositio}

\usepackage{amssymb,amscd, amsmath, caption, slashbox, multirow}

\newtheorem{thm}{Theorem}[section] 
\newtheorem{lem}{Theorem}[section]
\theoremstyle{definition} 
\theoremstyle{remark} 
\newtheorem{rem}{Remark}
\newtheorem{cor}{Corollary}

\DeclareMathOperator{\Spec}{\mathrm{Spec}}

 \newcommand{\cF}{\mathcal{F}}
 
 \newcommand{\cH}{\mathcal{H}}

 \newcommand{\cO}{\mathcal{O}}

 \newcommand{\F}{\mathbb{F}}
\newcommand{\G}{\mathbb{G}}

 \newcommand{\Q}{\mathbb{Q}}
 \newcommand{\R}{\mathbb{R}}

  \newcommand{\Z}{\mathbb{Z}}

  \def\arrow#1#2{\smash{\mathop{\longrightarrow}\limits^{#1}_{#2}}} 

\begin{document}

\title{Every Binary Self-Dual Code Arises From Hilbert Symbols}
\author{Ted Chinburg}
\email{ted@math.upenn.edu}
\address{Dept of Math\\Univ. of Pennsylvania\\209 S. 33rd Str\\Philadelphia\\PA19104-6395\\USA}
\author{Ying Zhang}
\email{yinzhang@sas.upenn.edu}
\address{Dept of Math\\Univ. of Pennsylvania\\209 S. 33rd Str\\Philadelphia\\PA19104-6395\\USA}
\dedication{\rm Oct 18, 2012}
\thanks{The first author was supported in part by  NSF Grant DMS1100355. The second author was
supported by a Benjamin Franklin fellowship from the University of Pennsylvania.}


\begin{abstract}
In this paper we construct binary self-dual codes using the \'etale cohomology of $\mu_2$ on the spectra of rings of $S$-integers of global fields. We will show that up to equivalence, all self-dual codes of length at least $4$ arise from Hilbert pairings on rings of $S$-integers of $\Q$.
This is an arithmetic counterpart of a result of Kreck and Puppe, who used cobordism
theory to show that
all self-dual codes arise from Poincar\'e duality on real three manifolds.
\end{abstract}

\maketitle

\section{Introduction}
\label{sec:intro}

Recently, M. Kreck and V. Puppe \cite{Kreck2008} gave a topological construction of all self-dual codes using  the cohomology of three-manifolds. A self-dual code is a triple $(W,V,E)$ in which $W$ is a vector space of finite even dimension over $\F_2$, $V$ is a subspace of $W$, $E $ is an ordered basis $\{e_i\}_{i = 1}^{2n}$ of $W$ and $V$ is its own orthogonal complement with respect to the
bilinear form $\langle \ , \ \rangle:W\times W \to \F_2$ defined by 
\[\langle \sum_{i=1}^{2n} a_i e_i, \sum_{i= 1}^{2n}b_i e_i \rangle =\sum_{i=1}^{2n} a_i b_i
\]
This implies that $V$ has dimension $n$.  In the following we will call the pair $(W,E)$ together with the form $\langle \  , \ \rangle$ a Euclidean space over $\F_2$, and $E$ is an orthonormal basis for $\langle \ , \ \rangle$.  Another self-dual code $(W',V',E')$ is defined
to be equivalent
to $(W,V,E)$ if there is a bijection between $E$ and $E'$ which when extended to an $\F_2$-linear isomorphism $W \to W'$
carries $V$ to $V'$.


The object of this note is to give a construction of self-dual codes which exploits the analogy between three-manifolds and the spectra of rings of $S$-integers of global fields.

Let $K$ be a global field of characteristic different from $2$. If $K$ is a function field, let $X$ be a smooth projective curve  with function field $K$. If $K$ is a number field, let $\cO_K$ be the ring of integers of $K$ and let $X= \Spec \cO_K$.  Suppose $v$ is a place of $K$ and that $\cF $ is a sheaf on the small \'etale site of $\Spec K_v$, where $K_v$ is the completion of $K$ at $v$. We define the reduced \'etale cohomology group $H^r_{et}(K_v,\cF)$ to be the usual \'etale cohomology group unless $v$ is real, in which case we let $H^r_{et}(K_v, \cF):=H^r_T(\Z/2, \cF)$ be the $r^{th}$ Tate cohomology of the $\mathrm{Gal}(\overline{K_v}/K_v) \cong \Z/2$ module associated to $\cF$. When $K$ is a number field, we let
$H^r_c(\Spec \cO_K,\cF)$ be the cohomology group with compact support defined by Milne   \cite[section 2, p165]{Milne2006}.

Let $S$ be a finite non-empty set of places of $K$ which contains all the archimedean places and all places of residue characteristic $2$. Let $U$ be the open complement of $S$ in $X$. We have a long exact sequence
\begin{equation}
\label{eq:longexact}
\cdots H^r_c(U, \cF)\to H^r_{et}(U, \cF)\to \oplus_{v\in S}H^r_{et}(K_v, i_v^*\cF)\arrow{\delta_r}{} H^{r+1}_c(U, \cF) \cdots
\end{equation}
where $i_v\colon\Spec K_v \to X$ is the canonical morphism.
We show the following result in
Section~\ref{sec:etale}.

\begin{thm}\label{thm:dual}
The image of the restriction homomorphism 
\[\Phi\colon H^1_{et}(U, \mu_2)\to \oplus_{v\in S} H^1_{et}(K_v, \mu_2)\]
 is its own orthogonal complement with respect to the non-degenerate bilinear product
\begin{equation}
\label{eq:pairing}
\left ( \oplus_{v \in S} H^1_{et}(K_v, \mu_2)\right ) \times \left ( \oplus_{v\in S} H^1_{et}(K_v, \mu_2)\right ) \to \oplus_{v \in S} H^2_{et}(K_v, \mu_2)\arrow{\delta_2}{} H^3_c(U, \mu_2) \cong \F_2
\end{equation}
which is the composition of the natural cup product pairing with the boundary map of (\ref{eq:longexact}) for $r = 2$.
\end{thm}

Therefore, if there is an orthonormal basis $E$ for the bilinear product (\ref{eq:pairing}) such that  $\oplus_{v \in S} H^1_{et}(K_v, \mu_2)$ is Euclidean with respect to $E$, then $image(\Phi)$ becomes a self-dual code by definition. This matter is addressed by the following result, which is also proved in Section~\ref{sec:etale}.

\begin{thm}\label{thm:euclid} If $v \in S$ is complex, then $H^i_{et}(K_v,\mu_2) = 0$ for all $i \ge 1$.
Otherwise,  there is a Euclidean basis for the cup product pairing
$$H^1_{et}(K_v, \mu_2) \times H^1_{et}(K_v,\mu_2) \to H^2_{et}(K_v,\mu_2) \cong \F_2$$ if and and only if $-1$ is not a square in $K_v$.  
\end{thm}

\begin{rem}\label{rem:euclid}
If $v$ is not complex, then $-1$ is not a square in $K_v$ when (i) $v$ is real, or 
(ii) $v$ is non-archimedean and the order of the residue field of $v$ is congruent to $3$ mod $4$, or (iii) $v$ is
non-archimedean of even residue characteristic and $-1 \not \in (K_v^*)^2$.
\end{rem} 

\begin{cor}
\label{cor:Euclidresult}
If every non-complex place $v$ of $S$ satisfies one of conditions (i) - (iii) of remark \ref{rem:euclid}, then
the union of the Euclidean bases produced by Theorem \ref{thm:euclid} gives a Euclidean
basis for the bilinear product space $\oplus_{v \in S} H^1_{et}(K_v, \mu_2)$ which is the orthogonal sum of the $H^1_{et}(K_v,\mu_2)$.  With respect to this basis the image of $\Phi$
in Theorem \ref{thm:dual} is a self-dual code.  This is the case, in particular, if $K = \Q$ and
every odd finite place $v$ in $S$ has residue field order congruent to $3$ mod $4$.
\end{cor}

Suppose $E$ is a basis for a finite dimensional space $W$ over $\F_2$ and that $\langle \ , \ \rangle$ is the associated
Euclidean bilinear form.  Let $n = \mathrm{dim}_{\F_2}(W)$.  The orthogonal group $O(n)$ is defined to be the group of linear transformations of $W$ 
which respect $\langle \ , \ \rangle$.  
The group $O(n)$ equals the group of permutations of the basis $E$ if and only
if $n \leq 3$.  Therefore, when $\mathrm{dim}(H^1_{et}(K_v,\mu_2)) \leq 3$ for all $v \in S$, the above orthonormal basis for $H^1_{et}(K_v, \mu_2)$ is unique up to permutations.  In fact, 
 $\mathrm{dim}(H^1_{et}(K_v,\mu_2)) > 3$ if and only if $K_v$ is a non trivial extension of $\Q_2$, see \cite[Proposition 5.7, Chap II]{Neukirch1999}.


Our second main result is the
following arithmetic analogue of Proposition 2 of \cite{Kreck2008}:

\begin{thm}
\label{thm:codedef}
 Up to equivalence, all self-dual codes of length
at least $4$ arise from the construction in Corollary \ref{cor:Euclidresult} when $K = \mathbb{Q}$.
In fact, each such code arises up to equivalence from infinitely many different subsets $S$
of the places of $K = \Q$.
 \end{thm}

To conclude this introduction, we give a more explicit description of the codes produced by Corollary \ref{cor:Euclidresult}
under the hypothesis that  $\mathrm{Pic}(U)$ has odd order.  This hypothesis is simply that the ring $O_{K,S}$
of $S$-integers of $K$ has class group of odd order.  In this case, $H^1_{et}(U,\mu_2)$ is isomorphic to $O_{K,S}^*/(O_{K,S}^*)^2$.  The group $H^1_{et}(K_v,\mu_2)$ is isomorphic to $K_v^*/(K_v^*)^2$.  The pairing
$$H^1_{et}(K_v,\mu_2) \times H^1_{et}(K_v,\mu_2) \to H^2_{et}(K_v,\mu_2) \subset \F_2$$ is  the Hilbert pairing
\begin{equation}
\label{eq:hilb1}
K_v^*/(K_v^*)^2 \times K_v^*/(K_v^*)^2 \to \{\pm 1\} \cong \F_2
\end{equation} 
(see \cite[chap. XIV]{Serre1968}).   The code space
$$\Phi(O_{K,S}^*/(O_{K,S}^*)^2) \subset \oplus_{v \in S} K_v^*/(K_v^*)^2$$
is simply the subgroup  which is the diagonal image of $O_{K,S}^*$ under the
natural homomorphism induced by the inclusion of $K$ into $K_v$ for $v \in S$.
When each non-complex $v$ satisfies one of the conditions in remark \ref{rem:euclid}, $K_v^*/(K_v^*)^2$ has a Euclidean basis. 
The Euclidean structure of the vector space $\oplus_{v \in S} K_v^*/(K_v^*)^2$ comes from the orthogonal
sum of the structures from each of the Hilbert pairings (\ref{eq:hilb1}).

When $K = \Q$
and $S = \{\infty , 2, p_1, \ldots, p_n\}$ for some positive primes $p_i \equiv 3$ mod $4$,  the
group $O_{K,S}^*$ is the group $\langle -1, 2, p_1,\ldots,p_n\rangle $.   The Hilbert
pairings in (\ref{eq:hilb1}) are easily described in this case (see  Section \ref{sec:Hibcode}).
For example, when $S=\{\infty, 2, 3, 7\}$, one generates the Hamming code $e_8$. When $S=\{\infty, 2, 7, 19, 31, 131, 179, 367,$ $883, 1223, 1307, 39079\}$, one gets the Golay code $g_{24}$.

In the course of proving Theorem \ref{thm:codedef}
in Section \ref{sec:Hibcode} we give a new parametrization of self-dual codes via matrices consisting
of $1 \times 2$ blocks which have certain properties (``boxed matrices").  The Theorem is proved by showing that all
self-dual codes are equivalent to codes which have boxed descriptions.   This has consequences
to the description of unimodular lattices, in view of the connection between such lattices
and self-dual codes proved in \cite{Kitazume1991}.

It would be very interesting
to see if boxed matrix descriptions of codes are also useful in the topological context
considered by Kreck and Puppe, e.g. in trying to construct explicitly the three manifolds
giving rise to self-dual codes.  At present, the construction of these manifolds is indirect and
proceeds by showing
that certain elements of cobordism groups are trivial. It would also be interesting if one could
see the proof of Theorem \ref{thm:codedef} for $K = \mathbb{Q}$ as a kind of explicit cobordism
calculation concerning the \'etale ``surfaces" $\Spec(K_v)$ inside the \'etale ``three-manifold" $\Spec(\Z)$.

\section{Etale Cohomology over Ring of Integers}\label{sec:etale}

We will first prove Theorem~\ref{thm:dual}, whose notations we now assume.


Artin-Verdier duality (c.f. \cite[section II.3]{Milne2006}) shows that
\begin{equation}
\label{eq:duality}
H^r_{et}(U,\mu_2(-1))\times H^{3-r}_c(U, \mu_2)\to H^3_c(U, \G_m)\cong \Q/\Z
\end{equation}
is a perfect duality of $\F_2$ vector spaces. Here $\mu_2(-1):=\cH om(\mu_2, \G_m)$ is canonically isomorphic to $\mu_2$. In the following we will not distinguish between these two sheaves.

For ease of notation, denote $A = H^1_{et}(U, \mu_2)$, $B = \oplus_{v\in S} H^1_{et}(K_v, \mu_2)$ and $C = H^2_c(U, \mu_2)$.
The pairing $B \times B \to \F_2$ which is the sum of the Hilbert symbols at $v$ for $v \in S$
is a perfect pairing by local class field theory.  This identifies the dual $\check B = \mathrm{Hom}_{\F_2}(B,\F_2)$
of $B$ with $B$.  By (\ref{eq:duality}) we have perfect pairing
$A \times C \to \F_2$ which identifies $\check A$ with $C$.   From (\ref{eq:longexact}) for $r = 1$ we have
an exact sequence
$$A \arrow{\Phi}{} B \arrow{\Psi}{} C$$
Here the above pairings identify $\Psi: B = \check {B}  \to C = \check{A}$ with the dual $\check{\Phi}$ of $\Phi$.
Hence
$$\mathrm{dim}(\mathrm{coker}(\Phi)) = \mathrm{dim}(\mathrm{ker}(\check \Phi))  = \mathrm{dim}(\mathrm{ker}(\Psi)) = \mathrm{dim}(\mathrm{image}(\Phi))$$
where the last equality follows from the above exact sequence. Thus
$\mathrm{dim}(\mathrm{image}(\Phi)) = \frac{1}{2} \mathrm{dim}(B)$, so all we now must show is
that $\mathrm{image}(\Phi)$ is self annihilating.  This is true for the following reasons.  The pairing $B \times B \to \F_2$
is given by the composition of the natural cup product pairing
$$\left (\oplus_{v\in S} H^1_{et}(K_v, \mu_2) \right ) \times \left (\oplus_{v\in S} H^1_{et}(K_v, \mu_2) \right )  \to
\left (\oplus_{v\in S} H^2_{et}(K_v, \mu_2) \right ) $$
with the boundary homomorphism
$$\oplus_{v\in S} H^2_{et}(K_v, \mu_2)  \arrow{\delta_2}{} H^3_c(U,\mu_2) \cong \F_2.$$
The pairing of two elements in the image of $A \to B$ is $0$ because the cup product of
such elements in $H^2_{et}(U,\mu_2)$ has trivial image under the composition of homomorphisms
$$H^2_{et}(U,\mu_2) \to \oplus_{v\in S} H^2_{et}(K_v, \mu_2)  \arrow{\delta_2}{} H^3_c(U,\mu_2)$$
in (\ref{eq:longexact}) when $r = 2$. This completes the proof.

We now prove Theorem \ref{thm:euclid}. A non-trivial vector space $W$ over $\F_2$ equipped with any non-degenerate bilinear product $\langle \ , \ \rangle$ has a Euclidean basis if and only if there is an $x \in W$ such that $\langle x, x\rangle=1$ is nontrivial in $\F_2$. This is easy to prove when $W$ has dimension less than or equal to $3$, and the general case follows by induction on dimension. 

Suppose now that $v$ is not a complex place.  The pairing
$$H^1_{et}(K_v, \mu_2) \times H^1_{et}(K_v,\mu_2) \to H^2_{et}(K_v,\mu_2) \cong \F_2$$
is the Hilbert pairing
$$(\ , \ )_v: K_v^*/(K_v^*)^2 \times K_v^*/(K_v^*)^2 \to \{\pm 1\} \cong \F_2.$$

This pairing has a Euclidean basis if and only if there is an element $\alpha \in K_v^*/(K_v^*)^2$
such that $(\alpha,\alpha)_v = -1 \in \pm 1$. Here $(\alpha,\alpha)_v =
(\alpha,-\alpha)_v \cdot (\alpha,-1)_v = (\alpha,-1)_v$.  By the definition of
the Hilbert pairing, $$(\alpha,-1)_v = \sigma(\sqrt{-1})/\sqrt{-1}$$
where $\sigma \in \mathrm{Gal}(K_v(\sqrt{-1})/K_v) = G$ is the image of $\alpha$
under the Artin map $\sigma: K_v^* \to  G$.  Hence there exists $\alpha$ with $(\alpha,-1)_v  = -1$ if and only
if $K_v(\sqrt{-1})$ is a non-trivial extension of $K_v$. This completes the proof of Theorem~\ref{thm:euclid}.

To conclude this section we make a few comments concerning the comparison of
the above construction with that of Proposition 2  of \cite{Kreck2008}.  The
code space considered by Kreck and Puppe is the image of the natural homomorphism
$$H^1(W,\F_2) \to H^1(\partial W,\F_2) = \oplus_{i = 1}^{2n} H^1(\R P^2,\F_2)$$
in which $W$ is a three-manifold with boundary  $\partial W$ the disjoint union
of $2n$ copies of $\R P^2$.  Each $\R P^2$ is the boundary of a three-orbifold
which is the quotient of a three dimensional ball $B^3$ by the antipodal involution which fixes
the center of $B^3$.   In the arithmetic context, the role of $\R P^2$ is played
by $\Spec(K_v)$, which has \'etale cohomological dimension $2$ when $v$ is
finite. The fixed loci of the involution should be compared to the spectrum of the residue field $\Spec(k(v))$. However, when $k(v)$ is the residue field of a finite place, $\Spec(k(v))$ has \'etale
cohomological dimension $1$ rather than $0$. So a better topological counterpart 
would make $\partial W$ a  finite disjoint union of connected smooth surfaces $S_i$, each of which
is the boundary of a three orbifold which is the quotient of a neighborhood
of a circle by an involution which fixes the circle. Klein bottles and two-dimensional tori could be realized as boundaries of such three orbifolds. When $S_i$ is a Klein bottle, $H^1(S_i,\F_2)$ is two-dimensional and the cup product
$H^1(S_i,\F_2) \times H^1(S_i,\F_2) \to H^2(S_i,\F_2) \cong \F_2$
has a Euclidean structure. Thus  a Klein bottle is analogous to $Spec(K_v)$ when $\# k(v) \equiv 3$ mod $4$. A two-dimensional torus is analogous to $Spec(K_v)$ when $\# k(v) \equiv 1$ mod $4$, since in this case the cup product pairing
on $H^1$ does not have a Euclidean structure.



\section{Hilbert Symbol Codes over $\mathbb{Q}$}
\label{sec:Hibcode}

Let $S$ be a finite set of places of $\mathbb{Q}$ consisting of the infinite place $\infty$,
the place determined by the prime $2$, and the places determined by a finite
set $p_1,\ldots,p_{n-2}$ of distinct positive prime numbers which are congruent to $3$ mod $4$.
The $S$-units $\mathbb{Z}_S^*$ of $\mathbb{Z}$ are then the subgroup $\langle -1, 2, p_1,\ldots,p_{n-2}\rangle $ of $\mathbb{Q}^*$
generated by $-1, 2, p_1,\ldots,p_{n-2}$. Recall that for each place $v_p \in S$
we have a Hilbert symbol pairing
$$(\ , \ )_{v_p}:\Q_{p}^*/(\Q_{p}^*)^2 \times \Q_{p}^*/(\Q_{p}^*)^2 \to \F_2.$$
Write the $\F_2$ vector space $W_p$ additively for the multiplicative group $\Q_{p}^* / (\Q_{p}^*)^2$. The space $W = \oplus_{v_p \in S} W_p$ is a finite dimensional vector space
over $\F_2$, and we have a non-degenerate pairing $(\ , \ ):W \times W \to \F_2$ defined
by $(\ ,\  ) = \sum_{v_p \in S}(\ , \ )_{v_p}$. Now we specify an explicit basis for each $W_p$ with respect to which the pairing $W\times W\to \F_2$ is Euclidean.

For an odd prime $p$ congruent to $3$ mod $4$, $-1$ is a non-square in $\Q^*_p$. We choose the representatives $\{-p, p\}$ in $\Q^*_p/(\Q^*_p)^2$ for the $\F_2$ basis for $W_p$.  For $\Q_2^*/(\Q_2^*)^2$ we use the representatives $\{-2, -10, -5\}$. For $\R^*/(\R^*)^2$ we use $-1$. It is an easy calculation to show that under this basis the Hilbert symbol pairing on $W$ is Euclidean, cf.  \cite[p. 23]{Serre1973}. 

Note that when $p$ is odd, a rational integer $l$ which is prime to $p$ is not a square in $\Q^*_p$ if and only if $l$ is a non-square mod $p$, and in this case the vector in $W_p$ corresponding to $l$ is $(1,1)$. If $l$ is a square in $\Q^*_p$, then the corresponding vector is $(0,0)$.

The image of $\Phi(\Z_S^*)$ in $W$ gives us a generator matrix $M$ of a linear code $V$ in $W$
which has the form indicated in Table \ref{tab:code}.  In this table 
there are three entries under $W_2$ because $\Q_2^*/(\Q_2^*)^2$ is a three
dimensional vector space over $\F_2$. The entries for a given
row under $W_2$ are the coefficients of the corresponding generator of $\Z_S^*$
in $\Q_2^*/(\Q_2^*)^2$ relative to the ordered basis $\{-2, -10,-5\}$. Under the chosen basis,

\begin{table}[ht]
\caption{A boxed code}
\label{tab:code}
\centering
\begin{tabular}{|c|llllll|}
\hline
 \multirow{2}{*}{\backslashbox{$S$-units}{places}}   & $W_{p_1}$     & $W_{p_2}$     & $\cdots$ & \multicolumn{2}{l}{$W_2$} & $W_{\R}$ \\
    & $\{-p_1, p_1\}$ & $\{-p_2, p_2\}$ & $\cdots$ & \multicolumn{2}{l}{$\{-2, -10, -5\}$} & $\{-1\}$ \\ \hline
 $p_1$ & $01$ & $00/11$
 &    & $00/11$
 & $1$ & $0$ \\
 $p_2$ & $11/00$
 & $01$ &    &   & $1$ & $0$ \\
 $\vdots$ & $\vdots$ &  & $\ddots$ &  & $\vdots$ & $\vdots$ \\
 $2$  &  $11/00$
 &  &  & $01$ & $1$ & $0$ \\
 $-1$ &  $11$ & $11$ & $\cdots$ & $11$ & $1$  & $1$ \\
 \hline
\end{tabular}
\end{table}
Theorem~\ref{thm:dual} guarantees that $V$ is a self-dual code in $W$. This can also be seen more directly in this special case by observing from Table~\ref{tab:code} that $M$ has rank $n$ and that $V$ is self annihilating by quadratic reciprocity.

We will view the $n\times 2n$ binary matrix $M$ in Table~\ref{tab:code} as an $n\times n$ block matrix $\tilde{M}$, where each block is a pair of elements $(a_{2i}, a_{2i+1})$. The matrix $\tilde{M}$
has the following properties:

\begin{flushleft}
\begin{itemize}
\item[(1)] The bottom row of $\tilde M$ has all entries equal to $(11)$.
\item[(2)] All entries
of the last column of $\tilde{M}$ equal the $(10)$ pair except for the $(11)$ in the final row.
\item[(3)] The diagonal elements of $\tilde{M}$ are all $(01)$ except for the final diagonal entry, which is equal to  $(11)$.
\item[(4)] All other pairs in $\tilde{M}$ are either $(00)$ or $(11)$, which we will call \emph{identical pairs}.
\end{itemize}
\end{flushleft}

We say that a block matrix having properties (1) - (4) is \emph{half-boxed}.  We will say that
$\tilde{M}$ is \emph{boxed} if the following is also true:
\begin{enumerate}
\item[(5)] For all $n-1 \ge i>j \ge 1$, $b_{ij} + b_{ji} = (11)$.
\end{enumerate}

It follows from quadratic reciprocity that a Hilbert code $V$ gives a boxed generator matrix $\tilde M$.
On the other hand, we can view the generator matrix of an arbitrary self-dual code $V'$ as an $n \times n$
matrix $\tilde M'$ whose entries are $1 \times 2$ blocks. The following lemma is an easy observation:

\begin{lem}\label{lem:box}
If $\tilde M'$ is half-boxed, and its row vectors are orthogonal to each other, then condition (5) is automatically satisfied, i.e. $\tilde M'$
is boxed.
\end{lem}

\begin{proof}[Proof of Theorem~\ref{thm:codedef}]

Every vector in a self-dual code $V$ must have even weight, i.e. an even number of $1$'s, since
every vector has trivial product with itself.  It
follows that $V$ must contain the  vector $m_1$ having all entries equal to $1$,
since this vector is orthogonal to all vectors of even weight.  Suppose now that
$M$ is the generator matrix for a self-dual code $V$ of length $2n$ and that the last
row of $M$ is $m_1$.  Observe that elementary row operations to $M$ correspond to a change of basis for the code $V$. Column permutations send $M$ to a generator matrix for a code equivalent to $V$. We will show by induction on $n$ that after applying a sequence
of invertible linear row operations and permutations of columns to $M$, one can make
the associated block matrix $\tilde{M}$ half-boxed.  We will in fact show that we can do
this without ever adding another row to the final row $m_1$ of $M$.  This will prove
the theorem, since the above operations lead to codes equivalent to $V$ by definition.

For $n = 2$ our claim is obvious.  We now suppose that $n > 2$ and that $M$ is the generator matrix for a self-dual code $V$ of length $2n$ and that the last
row of $M$ is $m_1$.  As $rank(M)=n$, the first row of $M$ is neither all-zero $00\cdots 0$ nor all-one $11\cdots 1$. Therefore we can permute the columns of $M$ to make the pair on the upper-left corner of $\tilde{M}$ equal $(01)$. We view $\tilde{M}$ as having four blocks:

\begin{table}[ht]
\caption{Block form of $\tilde{M}$}
\label{tab:block}
\centering
\begin{tabular}{|l|l|}
\hline
 $01$ & $u$ \\
 \hline
 $w$  & $M'$\\
 \hline
\end{tabular}
\end{table}
\noindent Here $w$ is a column block-vector of length $n-1$, and $u$ is a row block-vector of the same length. By adding the first row of $\tilde{M}$ to the $j$-th row if necessary, $2\leq j<n$, we can assume that $w$ consists only of identical pairs. Now $M'$ represents the generator matrix of a self-dual code of length $2n-2$ which has all $1$'s
in the final row. By our induction hypothesis, we can do column permutations and row operations on $M'$ such that $\tilde{M'}$ is in half-boxed form, where the bottom row of $M'$ remains all $1$'s. We perform these same operations
on the original matrix $M$.  This leads to a column block-vector $w$ in $\tilde {M}$ which still consists of identical pairs,
and the bottom row of $\tilde{M}$  remains $m_1$.

Consider the first $(n-2)$ pairs in $u$.  We have now arranged that the diagonal entries of $\tilde{M'}$
are all of the form $(01)$ except for the diagonal entry in the bottom row, and the entries of $\tilde{M'}$
which are not in the last row or column are identical pairs.  Therefore we can add to the first row of $\tilde{M}$
rows numbered $2$ through $n-1$ in such a way that all block entries of $u$ except for the last block become
identical pairs.   After these operations the upper-left corner of $M$ is either $01$ or $10$, since these operations amounts to adding certain identical pairs in $w$ to $01$.  Since the weight of the first row is even, the last pair of $u$ should be either $01$ or $10$. By adding the bottom row to the first row if necessary, we can assume the last pair of $u$ is $10$. Finally, if the upper-left corner of $M$ is $10$, we permute the first two columns of $M$ to make it $01$. Now the associated block matrix $\tilde{M}$ is in half-boxed form.  Therefore $\tilde{M}$ is in fact boxed by lemma \ref{lem:box}.

To complete the proof, we now need to show that every boxed matrix $\tilde{M}$
can be realized by the Hilbert code associated to some set $S = \{2,\infty,p_1,\ldots,p_{n-2}\}$.
To specify the $p_i$ we begin by requiring their classes in $\Q_2^*/(\Q_2^*)^2 \times \R^*/(\R^*)^2$ to be as in
the last two block columns of $\tilde{M}$.  This can be done with $p_i$ congruent to $3$ mod $4$.
We now choose the $p_i$ to lie in residue classes mod $p_j$ for $1 \le j < i \le n-2$ such
that the class of $p_i$ in $\Q_{p_j}^*/(\Q_{p_j}^*)^2$ is given by the entry $b_{ij}$ of $\tilde{M}$.
Since there is a unique boxed matrix which has these entries, and the block matrices
associated to Hilbert codes are boxed, we have now realized $\tilde{M}$ by a Hilbert code.
By the equidistribution of prime numbers in congruence classes, each self-dual code can be realized by this construction with infinite many distinct sets of places $S$.
\end{proof}

\begin{rem}
Suppose we specify arbitrary identical pairs for the entries $b_{ij}$ in a block matrix $\tilde{M}$
as $i$ and $j$ range over pairs for which $1 \le i < j \le n-1$.  Then $\tilde{M}$ can be completed
in a unique way to a boxed matrix.  This gives a new non-recursive way of writing down self-dual codes
of a given length. For some known recursive algorithms see \cite{Bilous2002} and \cite{Bouyuklieva2011}.
\end{rem}

\begin{rem}
Consider the case $K=\F_q(T)$, where $q$ is a prime power and $q \equiv 3$ mod $4$, $T$ is a parameter. $X=\mathbb{P}^1_{\F_q}$. $S=\{\frac{1}{T}, g_1(T),\cdots, g_{n-1}(T)\}$ where each $g_i(T)$ is a monic irreducible polynomial in $\F_q[T]$ of odd degree. Denote $W:=\oplus_{v\in S} K_v^*/(K_v^*)^2$. Then upon a suitable choice of basis, the Hilbert symbol pairing $W\times W\to \F_2$ is also Euclidean, and the diagonal image of $\Phi: \cO^*_U/(\cO^*_U)^2=\langle -1, g_1(T), \cdots, g_{n-1}(T)\rangle$ in $W$ is also given by a boxed matrix.
\end{rem}

\bibliographystyle{alpha}
\bibliography{ref3}

\end{document}